\documentclass{amsart}
\usepackage{a4wide, amsmath, amssymb, url}

\usepackage[colorlinks]{hyperref}
\hypersetup{linkcolor=blue, urlcolor=blue, citecolor=red}

\newtheorem{thm}{Theorem}[section]
\newtheorem{lem}[thm]{Lemma}
\newcommand{\proofref}[1]{\noindent {\emph{Proof of Theorem} \ref{#1}.\ }}

\newcommand{\nat}{\mathbb{N}}
\newcommand{\reals}{\mathbb{R}}
\newcommand{\set}[2]{\ensuremath{\{ #1  : #2 \}}}
\newcommand{\genset}[1]{\ensuremath{\langle #1 \rangle}}

\newcommand{\fin}{\mathfrak{F}}
\newcommand{\jas}{^{\smallfrown}}
\newcommand{\C}{\mathcal{C}}
\newcommand{\A}{\mathcal{A}}
\newcommand{\ntrans}{\nat^{\nat}}
\newcommand{\U}{\mathfrak{U}}
\renewcommand{\S}{\mathcal{S}}

\renewcommand{\to}{\longrightarrow}

\title[The Bergman-Shelah Preorder on Transformation Semigroups]{The Bergman-Shelah Preorder on Transformation Semigroups}

\author{Z. Mesyan, J. D. Mitchell, M. Morayne, and Y. P\'eresse}

\address{Z. Mesyan, Department of Mathematics, University of Colorado, 
Colorado Springs, CO 80918, USA}
\email{\href{mailto:zmesyan@uccs.edu}{zmesyan@uccs.edu}}
\address{J. D. Mitchell, Mathematical Institute, North Haugh, St Andrews, Fife, KY16 9SS, Scotland}
\email{\href{mailto:jdm3@st-and.ac.uk}{jdm3@st-and.ac.uk}}
\address{M. Morayne,  Institute of Mathematics  and Computer Science, Wroc\l aw University of Technology, Wybrze\.ze Wyspia\'nskiego 27, 50-370 Wroc\l aw, Poland}
\email{\href{mailto:michal.morayne@pwr.wroc.pl}{michal.morayne@pwr.wroc.pl}}
\address{Y. P\'eresse, Mathematical Institute, North Haugh, St Andrews, Fife, KY16 9SS, Scotland}
\email{\href{mailto:yhp1@st-and.ac.uk}{yhp1@st-and.ac.uk}}


\begin{document}

\begin{abstract}
Let $\nat^\nat$ be the semigroup of all mappings on  the natural numbers $\nat$, and let $U$ and $V$ be 
subsets of $\nat^\nat$. We write $U\preccurlyeq V$ if there exists a countable subset $C$ of $\nat^\nat$ such that 
$U$ is contained in the subsemigroup generated by $V$ and $C$. We give several results about  the structure of 
the preorder $\preccurlyeq$. In particular, we show that a certain
statement about this preorder is equivalent to the Continuum Hypothesis.

The preorder $\preccurlyeq$ is analogous to one introduced by Bergman and Shelah  on subgroups of the symmetric group on $\nat$. 
The results in this paper suggest that the preorder on subsemigroups of $\nat^\nat$ is much more complicated than that on subgroups of the symmetric group.
 
\end{abstract}

\maketitle


\section{introduction and background}

The semigroup of all mappings from $\nat=\{0,1,2,\ldots\}$ to itself is denoted by $\nat^\nat$.  Given subsets $U$ and $V$  of $\nat^\nat$, we write $U\preccurlyeq V$ if there exists a countable subset $C$ of $\nat^\nat$ such that $U$ is contained in the subsemigroup $\genset{V, C}$ generated by $V$ and $C$.  It follows from a classical result by Sierpi\'nski \cite{Banach1935aa} that if $U\preccurlyeq V$, then there exist $f,g\in \ntrans$ such that $U \subseteq \genset{V,f,g}$. So replacing the word `countable' above by `finite' or even `2-element' yields an equivalent definition of $\preccurlyeq$.
We write $U\approx V$ if  $U\preccurlyeq V$ and $V\preccurlyeq U$, and we write $U\prec V$ if $U\preccurlyeq V$ and  $U \not\approx V$.

The semigroup $\nat^\nat$ has a natural topology: the product topology arising from the discrete topology on $\nat$; see \cite[Section 9.B(7)]{Kechris1995aa} for further details.  Under this topology, composition of functions is continuous, making $\nat^\nat$  a \emph{topological semigroup}.  Let $S_{\infty}$ denote the symmetric group on $\nat$, i.e. the group of invertible elements of $\nat^{\nat}$. As it happens the function $x\mapsto x^{-1}$ on $S_{\infty}$ is also continuous, and so $S_{\infty}$ is a \emph{topological group} with the induced topology.
We refer to subgroups of $S_{\infty}$ and subsemigroups of $\nat^\nat$ that are closed in the relevant topologies as \emph{closed subgroups} and \emph{closed subsemigroups}, respectively.
It is a well-known fact that  the closed subsemigroups of $\nat^\nat$ are precisely the endomorphism semigroups of relational structures on $\nat$ and that  the closed subgroups of $S_{\infty}$ are the corresponding automorphisms groups; see, for example,  \cite[Theorem 5.8]{Cameron1999aa}.  

The preorder $\preccurlyeq$ is analogous to a preorder on the subsets of $S_{\infty}$ introduced in \cite{Bergman2006ac}: 
if $U,V \subseteq S_{\infty}$, then  $U$ is less than $V$ whenever $U$ is contained in the subgroup generated by $V\cup C$ for some countable $C\subseteq S_{\infty}$. Once again, insisting that $C$ is finite, or even of size $2$, yields an equivalent definition; see \cite[Lemma 3(i)]{Bergman2006ac}.
In \cite{Bergman2006ac} it is shown that the closed subgroups of $S_{\infty}$ fall into 
 four equivalence classes with respect to this preorder. Various classes of subsemigroups of $\nat^\nat$ are classified according to $\approx$ in \cite{Mesyan2007aa} and \cite{Mitchell2010ac}.  
The situation is much more complicated in $\nat^\nat$, as in particular, there are infinitely many distinct $\approx$-classes containing closed subsemigroups. 
For example,  define for each $n\geq 2$ 
 $$\fin_n=\set{f\in \nat^\nat}{|f(\nat)|\leq n}.$$ 
 It is straightforward to show that $\fin_n$ is a closed subsemigroup of $\nat^\nat$ for all $n\geq 2$. Furthermore each $\fin_n$ is an ideal of $\ntrans$ and so if $U\subseteq \genset{ \fin_n, C}$ for some $U,C\subseteq \ntrans$, then $U\setminus \fin_n \subseteq \genset{C}$. Hence $U \preccurlyeq \fin_n$ if and only if $U\setminus \fin_n$ is countable. But $|\fin_m \setminus \fin_n|=2^{\aleph_0}$ whenever $m>n$ and so
$\fin_2\prec \fin_3 \prec \cdots.$

We prove five results that exhibit the complicated structure of $\preccurlyeq$ and its sensitivity to set-theoretic assumptions. 

In Theorem \ref{CHequivalent}, we show that the Continuum Hypothesis holds if and only if   there exists a subsemigroup $S$ of $\nat^\nat$ such that $S\approx \nat^\nat$ and for all subsemigroups $T$ of $S$ either $T\approx \nat^\nat$ or $T$ is equivalent to the trivial semigroup $\{1_{\nat}\}$. 
We prove that for every closed subsemigroup $S$ of $\nat^\nat$ with cardinality $2^{\aleph_0}$ there is a closed subsemigroup $T$  of $\fin_2$ of cardinality $2^{\aleph_0}$ such that $T \preccurlyeq S$ (Theorem \ref{cantor_bendix}).  Theorem \ref{cantor_bendix} could be viewed as an analogue of the classical theorem that every perfect Polish topological space contains a copy of the Cantor set. 
To show that $T$ in Theorem \ref{cantor_bendix} cannot be replaced by $\fin_2$, we associate a semigroup to each almost disjoint family of subsets of $\nat$ with cardinality $2^{\aleph_0}$ and show that any such semigroup  is incomparable to $\fin_n$ for all $n\in \nat$ (Theorem \ref{lost_monoid}). We prove that there are anti-chains of $\approx$-classes containing closed subsemigroups of $\nat^\nat$ with arbitrary finite length  (Theorem \ref{finite_preorder}). 
Finally, we show that there exists a chain  of $\approx$-classes with length $\aleph_1$ containing (not necessarily closed) subsemigroups of  $\fin_2$ (Theorem \ref{disjoint}), establishing a new lower bound for the number of $\approx$-classes. 

 It seems unlikely that a usable classification of $\approx$-classes and the partial order induced by $\preccurlyeq$ can be found. However, further potentially tractable questions about the structure of $\preccurlyeq$ are, as yet, unanswered. 
For instance, what is the number of $\approx$-classes? What is the number of $\approx$-classes containing closed subsemigroups? Which preorders can be embedded in $\preccurlyeq$? More specifically, does there exist an infinite anti-chain or an infinite descending chain? Do there exist $U, V\leq \nat^\nat$ such that $U\prec V$ and whenever $U\preccurlyeq W\preccurlyeq V$ either $W\approx U$ or $W\approx V$?  


\section{Continuum Hypothesis}

The Continuum Hypothesis is the statement: $\aleph_1=2^{\aleph_0}$. G\"odel \cite{Godel1940aa} and Cohen \cite{Cohen1963aa}, \cite{Cohen1964aa} showed that it is independent of the standard axioms of set theory (ZFC).
The Continuum Hypothesis is equivalent to the existence of an uncountable family $\mathcal{F}$ of analytic functions from $\mathbb{C}$ to $\mathbb{C}$ satisfying
$$|\set{f(x)}{f\in \mathcal{F}}|\leq \aleph_0$$
for all $x\in \mathbb{C}$, as well as the existence of a function $f=(f_1,f_2)$ from $\reals$ onto $\reals^2$ such that for all  $x\in \reals$ either
$f_1$ or $f_2$ is differentiable at $x$ (see \cite{Erdos1964aa} and  \cite{Morayne1987aa}, respectively). 
For more information on the history of the Continuum Hypothesis see  \cite{Sierpinski1934ab} or \cite{Steprans201aa}.

In some sense, the above results are analytic versions of the Continuum Hypothesis; in this section we present an algebraic version.

\begin{thm}\label{CHequivalent}
The following are equivalent:
\begin{itemize}
\item[(i)] the Continuum Hypothesis;
\item[(ii)] there exists a subsemigroup $S$ of $\nat^\nat$ such that $S\approx \nat^\nat$ and for all subsemigroups $T$ of $S$ either $T\approx \nat^\nat$ or $T\approx \{1_{\nat}\}$.
\end{itemize}
\end{thm}

We require two lemmas to prove Theorem \ref{CHequivalent}. 
The proof of the first is essentially Banach's argument \cite{Banach1935aa} for Sierpi\'nski's theorem  in
\cite{Sierpinski1935aa}. 

\begin{lem}\label{banach_yet_a_fucking_gain}
Let $f\in \nat^\nat$ be any injective function with $|\nat\setminus f(\nat)|=|\nat|$ and let 
$g_1, g_2, \ldots\in \nat^\nat$ be arbitrary. Then there exists $h\in \nat^\nat$ such that 
$g_1, g_2, \ldots\in \genset{f, h}$.
\end{lem}
\proof 
Let $X_0=\nat\setminus f(\nat)$ and let $X_i=f^i(X_0)$ for all $i>0$. 
Then  clearly $X_0\cap X_i=\emptyset$ for all $i>0$. Hence if $k>j$, 
$X_j\cap X_k=f^{j}(X_{0})\cap f^{j}(X_{k-j})=\emptyset$, since $f$ is injective. It follows that $X_0, X_1, \ldots$ are disjoint infinite subsets of $\nat$. 

Let $X_{0,0}, X_{0,1}, X_{0,2},\ldots$ be sets partitioning $X_0$ such that $|X_{0,0}|=|\nat\setminus \bigcup_{i=0}^{\infty} X_i|$ and $|X_{0,i}|=|\nat|$ for all $i>0$. 
We also let  $h$ be any map taking $\nat\setminus \bigcup_{i=0}^{\infty} X_i$ bijectively to $X_{0,0}$ and $X_i$
bijectively to $X_{0,i}$ for all $i>0$.  It is straightforward to verify that $hf^ihf$ maps $\nat$ bijectively to $X_{0, i}$ for all $i>0$.  Since $h$ is not yet defined on $X_0$, we can define it by:
$$h(n)=g_i\big((hf^ihf)^{-1}(n)\big)$$
for all $n\in X_{0,i}$ and for all $i>0$ and $h$ can be defined arbitrarily on $X_{0,0}$. 

It is easy to verify that $g_i=h^2f^ihf$ for all $i>0$. \qed

\begin{lem}\label{zero}
Let $\gamma$ be an ordinal and for every $\alpha<\gamma$ let $u_{\alpha}\in \nat^\nat$. Then there exist   $h, k \in \nat^\nat$ and for every $\alpha<\gamma$ there is a mapping $g_{\alpha}\in \nat^\nat$  such that:
\begin{itemize}
\item[(i)] $g_\alpha g_{\beta}$ is the constant function with value $0$ for all $ \beta<\gamma$;
\item[(ii)] $u_{\alpha}=kg_{\alpha}h$.
\end{itemize}
\end{lem}
\proof Let $X$ be any infinite coinfinite subset of $\nat$ such that $0\not\in X$, let $h:\nat\to X$ be any bijection, and let 
$k\in \nat^\nat$ be any function mapping $\nat\setminus X$ bijectively to $\nat$. Then for all $\alpha<\gamma$ define
$g_\alpha\in \nat^\nat$ by 
\begin{equation*}
g_{\alpha}(n)=
\begin{cases}
(k|_{\nat\setminus X})^{-1}u_{\alpha}h^{-1}(n)&\text{if } n\in X\\
0&\text{if } n\not\in X,
\end{cases}
\end{equation*}
where $k|_{\nat\setminus X}$ denotes the restriction of $k$ to $\nat\setminus X$.  The mappings
 $h,k$, and $g_{\alpha}$ ($\alpha<\gamma$) have the required properties. \qed\vspace{\baselineskip}

\proofref{CHequivalent} 
(i) $\Rightarrow$ (ii). Write $\nat^\nat=\set{f_{\alpha}}{\alpha<\aleph_1}$ and let $f\in \nat^\nat$ be an injection 
such that  $\nat\setminus f(\nat)$ is infinite. We define  a subset $U=\set{u_{\alpha}}{\alpha<\aleph_1}$ of $\nat^\nat$ such that every uncountable subset $V$  of $U$ satisfies $V\approx \nat^\nat$. 
Set $u_0=f_0$. If $\alpha<\aleph_1$ and $u_{\beta}$ is defined for all $\beta<\alpha$, then, by Lemma \ref{banach_yet_a_fucking_gain}, there exists $u_\alpha\in\nat^\nat$ such that 
$$\set{f_{\beta}}{\beta<\alpha}\subseteq \genset{f, u_{\alpha}}.$$

If $V$ is any uncountable subset of $U$, then for all $\beta<\aleph_1$ there exists $\lambda(\beta)$ such that $\beta<\lambda(\beta)<\aleph_1$ and $u_{\lambda(\beta)}\in V$. It follows that 
$f_{\beta}\in \genset{f, u_{\lambda(\beta)}}\subseteq \genset{f, V}$ for all $\beta<\aleph_1$ and so 
$\nat^\nat \subseteq \genset{f, V}$.  In particular, $V\approx \nat^\nat$. 

 Applying Lemma \ref{zero} to $U$ and $\gamma=\aleph_1$ we obtain $g_{\alpha}\in \nat^\nat$ for all $\alpha<\aleph_1$ and $h,k\in \nat^\nat$ with the properties given in the lemma.  
We set $S$ to be the semigroup consisting of $\set{g_{\alpha}}{\alpha<\aleph_1}$ and the constant mapping with value $0$.  To verify that $S$ satisfies (ii), let $T$ be any subset of $S$. If $T$ is uncountable, then $\genset{T, h,k}$ contains an uncountable subset of $U$ and so $T\approx \nat^\nat$ from above. If $T$ is countable, then $T\approx \{1_{\nat}\}$, by definition. \vspace{\baselineskip}

(ii) $\Rightarrow$ (i). Let $T$ be any subset of $S$ such that  $|T|=\aleph_1$. Then, by assumption,  $T\approx \nat^\nat$ and so $2^{\aleph_0}=|\nat^\nat|=|T|=\aleph_1$, as required. \qed


\section{The structure under $\mathfrak{F}_2$}

The following theorem suggests that to understand the structure of $\preccurlyeq$ we should first understand its structure on subsemigroups of $\fin_2$. 

\begin{thm}\label{cantor_bendix}
Let $S$ be a closed subsemigroup of $\nat^\nat$ of cardinality $2^{\aleph_0}$. Then there exists a closed subsemigroup $T$ of $\fin_2$ such that $|T|=2^{\aleph_0}$ and $T\preccurlyeq S$.
\end{thm}




We follow the convention that if $n\in\nat$, then  $n=\{0,1,\ldots, n-1\}$.
Let $\mathcal{C}=2^{\nat}$ denote the Cantor set (i.e., all functions from $\nat$ to $\{0,1\}$). Then it is straightforward to prove that $\C\approx \fin_2$.  

For a subset $A$ of $\nat$, we denote the set of  finite sequences of elements of $A$  by $A^{<\nat}$  and we write $x=(x(0), x(1), \ldots, x(n-1))$. 
The length $n$ of $x$ is denoted by $|x|$, and 
we define $$x\jas m=(x(0), x(1), \ldots, x(n-1), m)\quad\text{ where }m\in \nat.$$

If $f\in\nat^\nat$,  then we denote the restriction $(f(0), f(1), \ldots, f(m-1))$ of $f$ to the set $m={\{0,1,\ldots,m-1\}}$ by $f|_m$.
Similarly, if  $x\in \nat^{<\nat}$ and $|x|\geq m$, then $x|_m=(x(0), x(1), \ldots, x(m-1))$.





The proof of the following lemma is similar to that of the fact that every perfect Polish space contains a copy of the Cantor set given in \cite[Theorem 6.2]{Kechris1995aa}.

\begin{lem}\label{perfecter} Let $S$ be a closed subset of $\ntrans$ with $|S|=2^{\aleph_0}$. Then there exist $U\subseteq S$ and $f \in \mathcal{C}$ such that $U$ is homeomorphic to $\mathcal{C}$ and 
the map $\lambda: U \longrightarrow \mathcal{C}$ defined by $\lambda(g)=f\circ g$ for all $g\in U$ is a homeomorphism from $U$ to $\lambda(U)$.
\end{lem}
\begin{proof}
By assumption, $S$ is a closed subset of $\nat^\nat$, and so $S$ is a Polish space. Since $|S|=2^{\aleph_0}$,  the Cantor-Bendixson Theorem \cite[Theorem 6.4]{Kechris1995aa} implies that there exists  a perfect subset  of $S$, i.e. a closed set with no isolated points.  Assume without loss of generality that $S$ is perfect. Let 
$$\S=\set{f|_n\in \nat^{<\nat}}{n\in\nat,\ f\in  S}$$
(the set of finite restrictions of elements in $S$),  and if $x\in \S$, then define 
$$[x]_\S=\set{y\in \S}{y|_{|x|}=x}$$
(the set of finite extensions of $x$ in $\S$), and  
 $$E(x)=\set{i\in \nat}{x\jas i\in \S}.$$ 

We begin by showing that there exist $\iota_0, \iota_1: 2^{<\nat}\to \nat$ and $\sigma:2^{<\nat}\to \S$ such that 
\begin{enumerate}
\item[\rm (i)] $\iota_0(2^{<\nat})\cap \iota_1(2^{<\nat})=\emptyset$;
\item[\rm (ii)] $\sigma(x\jas j)\in[\sigma(x)\jas \iota_j(x)]_\S$ for all $j\in\{0,1\}$
and for all $x\in 2^{<\nat}$.
\end{enumerate}

Since $S$ is perfect, for every $x\in \S$  there exists $y\in [x]_\S$  such that $|E(y)|\geq 2$.
 There are two cases to consider.\vspace{\baselineskip}

\noindent{\bf Case 1.} \emph{there exist finite $A\subseteq \nat$ and  $x\in \mathcal{S}$    such that  
every $y\in [x]_\S$ satisfying $|E(y)|\geq 2$ also has the property that 
$E(y)\subseteq A$.}

Assume without loss of generality that the set $A$ is minimal with the above property, i.e. for every $B \subsetneq A$ and 
$z\in\S$ there exists $y\in  [z]_\S$ such that  $|E(y)|\geq 2$, and 
$E(y) \not \subseteq B$.  
Let $a\in A$ be arbitrary but fixed. Then for all $x'\in  [x]_\S$  there exists $y\in [x']_\S$ such that 
$|E(y)|\geq 2$ and $E(y)\not \subseteq A\setminus \{a\}$. But  $E(y)\subseteq A$ and so $a\in E(y)$. We have shown that:
\begin{enumerate}
\item[$(\star)$] for all $x'\in  [x]_\S$  there exists $y\in [x']_\S$ such that $|E(y)|\geq 2$ and $a\in E(y)$. 
\end{enumerate} 
  We  now use ($\star$) to recursively define $\iota_0, \iota_1: 2^{<\nat}\to \nat$ and $\sigma:2^{<\nat}\to \nat^{<\nat}$ satisfying (i) and (ii) above. As a first step, let $\sigma(\emptyset)\in  [x]_\S$ be any element such that  $|E(\sigma(\emptyset))|\geq 2$ and $a\in E(\sigma(\emptyset))$.  
 
 Assume that $\sigma(z)\in [x]_\S$ is defined for some $z\in 2^{<\nat}$ such that $|E(\sigma(z))|\geq 2$ and $a\in E(\sigma(z))$. Define $\iota_0(z)=a$ and $\iota_1(z)$ to be any element in $E(\sigma(z))\setminus \{a\}$. By ($\star$) we can define  $\sigma(z\jas j)\in  [\sigma(z)\jas \iota_j(z)]_\S$   such that $|E(\sigma(z\jas j))|\geq 2$ and $a\in E(\sigma(z\jas j))$ for $j\in \{0,1\}$.
 \vspace{\baselineskip}

\noindent{\bf Case 2.} \emph{for all finite $A\subseteq \nat$ and for all $x\in \S$ there exists $y\in [x]_{S}$ with $|E(y)|\geq 2$ but $E(y)\not\subseteq A$.}

List the elements of $2^{<\nat}$ as $x_0, x_1, \ldots$ in any way such that $|x_j|< |x_k|$ implies $j< k$. Let $\sigma(x_0)\in \S$ be such that $|E(\sigma(x_0))|\geq 2$ and let $\iota_0(x_0),\iota_1(x_0)\in E(\sigma(x_0))$ be such that $\iota_0(x_0)\not=\iota_1(x_0)$.
Assume that for all $j<k$ we have already defined 
$\sigma(x_j)$, $\iota_0(x_j)$, and $\iota_1(x_j)$ such that 
$\sigma(x_j)\jas \iota_i(x_j)\in \S $ for $i\in\{0,1\}$.
Set $A_k= \set{\iota_0(x_l), \iota_1(x_l)}{l<k}$.  Write $x_k=\mbox{$x_j$}\jas r$ for some $r\in\{0,1\}$ and $j\in\nat$. Then $j<k$  from the order on the elements of $2^{<\nat}$. Hence by the assumption of this case there exists $\sigma(x_k)\in [\sigma(x_j)\jas \iota_r(x_j)]_{S}$  such that $|E(\sigma(x_k))|\geq 2$ and $E(\sigma(x_k))\not\subseteq A_k$. Let $m,n \in E(\sigma(x_k))$ be such that $m\not =n$ and $m\not\in A_k$. If $n\in A_k$, then $n=\iota_l(x_j)$ for some $j<k$ and some $l\in\{0,1\}$. In this case, set $\iota_l(x_k)=n$ and set $\iota_{l+1\pmod 2}(x_k)=m$. If $n\not\in A_{k}$, then set $\iota_0(x_k)=m$ and $\iota_1(x_k)=n$. 
\vspace{\baselineskip}

In either case,  the functions $\iota_0, \iota_1$ and $\sigma$ have the required properties. 

We will now use $\iota_0, \iota_1$ and $\sigma$ to define $U$ and $f$. If $x\in \C$, then $\bigcap_{n\in\nat} [\sigma(x|_n)]$ is a singleton in $S$ since $S$ is closed and hence complete.  
 Let  $\{\Psi(x)\}=\bigcap_{n\in\nat} [\sigma(x|_n)]$. Then $\Psi:\C\to P$ is a homeomorphism from $\C$ to $\Psi(\C)$ and we set $U=\Psi(\C)$.  Let $f\in \mathcal{C}$ be any mapping  such that 
\begin{equation*}
f(m)=
\begin{cases}
0&\text{if } m\in \iota_0(2^{<\nat})\\
1&\text{if } m\in \iota_1(2^{<\nat}).
\end{cases}
\end{equation*}
Then $\lambda: U\to \C$ defined by $\lambda(g)=f\circ g$ is continuous, since $\nat^{\nat}$ is a topological semigroup.  It only remains to prove that $\lambda$ is injective. Let $\Psi(x),\Psi(y)\in U=\Psi(\C)$ such that $\Psi(x)\not=\Psi(y)$. Then, without loss of generality,  there exist $m\in\nat$ and $z\in 2^{<\nat}$ such that $x|_m=z\jas 0$ and $y|_m=z\jas 1$. 
It follows that  $\sigma(z)\jas \iota_0(z)$ is a restriction of $\sigma(x|_m)=\sigma(z\jas 0)$ and  $\sigma(z)\jas \iota_1(z)$ is a restriction of $\sigma(y|_m)=\sigma(z\jas 1)$. The number  $|\sigma(z)|$ is in the domain of  $\sigma(z)\jas \iota_0(z)$  and hence of $\sigma(x|_m)=\sigma(z\jas 0)$ and so 
$$
\Psi(x)\big(|\sigma(z)|\big)=\sigma(x|_m)\big(|\sigma(z)|\big)=\sigma(z\jas 0)\big(|\sigma(z)|\big)
=\big(\sigma(z)\jas \iota_0(z)\big)\big(|\sigma(z)|\big)=\iota_0(z).
$$
Hence 
$$\lambda(\Psi(x))(|\sigma(z)|)=(f\circ \Psi(x))(|\sigma(z)|)=f\big(\Psi(x)(|\sigma(z)|)\big)=f(\iota_0(z))=0.$$
Likewise, $\Psi(y)\big(|\sigma(z)|\big)=\iota_1(z)$ and so $\lambda(\Psi(y))\big(|\sigma(z)|\big)=f(\iota_1(z))=1.$
Therefore $\lambda(\Psi(x))\not=\lambda(\Psi(y))$ and so $\lambda$ is injective. 
\end{proof}

\proofref{cantor_bendix} 
Let $S$ be a closed subsemigroup of $\ntrans$ with $|S|=2^{\aleph_0}$. Then, by Lemma \ref{perfecter}, there exist $U\subseteq S$ and $f \in \mathcal{C}$ such that $U$ is homeomorphic to $\mathcal{C}$ and 
the map $\lambda: U \longrightarrow \mathcal{C}$ defined by $\lambda(g)=f\circ g$ for all $g\in U$ is a homeomorphism from $U$ to $\lambda(U)$.
 
Then $\lambda(U)$, being homeomorphic to $\C$, is compact. Hence, since $\nat^\nat$ is Hausdorff, $\lambda(U)\subseteq \C$ is closed. Let $T$ be the subsemigroup generated by $\lambda(U)$, the transposition $(0\ 1)\in S_{\infty}$, and the constant function with value $0$. Then $T$ is the union of $\lambda(U)$,  $\set{(0\ 1)\circ \lambda(u)}{u\in U}$, and the constant functions with value $0$ and $1$. In particular, $T\leq \fin_2$  and $T$ is closed (being the finite union of closed sets). Also  $|T|=2^{\aleph_0}$ and $T\approx \lambda(U)$.  Furthermore, $\lambda(U)=\set{f\circ g}{g\in U}\subseteq \genset{U,f}$ and so
 $T\approx \lambda(U) \subseteq \genset{U,f}\approx U\subseteq S$. In particular, $T\preccurlyeq S$.
\qed


\section{Almost disjoint families}\label{almost_disjoint_section}

If $A$ is a subset of $\nat$, then define $s_A\in \nat^\nat$ by 
\begin{equation}
s_A(n)=
\begin{cases}
n&\text{if } n\in A\\
0&\text{if } n\not\in A.\\
\end{cases}
\end{equation}
The \emph{power set} of $A\subseteq \nat$ is denoted by $\mathcal{P}(A)$. 
If $\A\subseteq \mathcal{P}(\nat)$, then set
\begin{equation}\label{s_a}
S_{\A}=\set{s_A\in\nat^\nat}{A\in \A}.
\end{equation}
Note that $S_{\A}$ is a subsemigroup of $\nat^\nat$ if and only if $\A$ is closed under taking finite intersections.  

A set $\A$ of subsets of $\nat$ is called \emph{almost disjoint} if $A\cap B$ is finite for all $A,B\in \A$. It is not hard to show that there exist almost disjoint $\A$ such that $|\A|=2^{\aleph_0}$; see, for example, \cite[Theorem 1.3]{Kunen1983aa}.  
Let 
$$\fin=\bigcup_{n\in\nat} \fin_n.$$ 
In this section we prove the following theorem.

\begin{thm} \label{lost_monoid} 
If $\mathcal{A}$ is an almost disjoint family of cardinality $2^{\aleph_0}$,  then $S_{\A}$ is 
incomparable under $\preccurlyeq$ to $\fin$ and $\fin_n$ for all $n\geq 2$.
\end{thm}
 If we identify $\mathcal{P}({\nat})$ with $2^{\nat}$ equipped with the product topology, then the function $A\mapsto s_A$ is a homeomorphism from $2^{\nat}$ to $S_{\mathcal{P}({\nat})}$. Thus, since $2^{\nat}$ is compact,  $S_{\mathcal{P}({\nat})}$ is closed in $\nat^{\nat}$ and so $S_{\A}$ is closed in $\nat^{\nat}$ if and only if $\A$ is closed in $2^{\nat}$.
  For example, if $\A$ is the almost disjoint family defined as the infinite paths starting at the root of an infinite binary tree labelled by the natural numbers (without repeats), then $S_{\A}$ is closed. 
Hence, by Theorem \ref{cantor_bendix}, there exists $T\preccurlyeq S_{\A}$ such that $\{1_{\nat}\}\prec T\preccurlyeq \fin_2$. Note that Theorem \ref{lost_monoid} implies that the semigroup $T\not\approx \fin_2$, and so, in general, $T$ in Theorem \ref{cantor_bendix} cannot be replaced by $\fin_2$.

Throughout the remainder of this section we use $\A$ to denote  an arbitrary almost disjoint family of cardinality $2^{\aleph_0}$. 


Let $X$ and $Y$ be countably infinite sets and let $f,g:X\to Y$. Then we say that $f$ is \emph{almost injective} if it is injective on a cofinite subset of $X$.  If all but finitely many elements of $X$ are contained in $Y$, then we say that $X$ is \emph{almost contained in} $Y$.  If $f$ and $g$ agree on a cofinite subset of $X$, then we say that $f$ and $g$ are \emph{almost equal}. 

\begin{lem}\label{2choices}
Let $u_0, \ldots, u_{r}\in \nat^\nat$ and let $N$ be an infinite subset of $\nat$  such that $u_{r-1}\cdots u_0$ is almost injective on $N$ and $u_j\cdots u_0(N)$ is almost contained in some $A(j)\in \mathcal{A}$ for all $j\in\{0,\ldots, r-1\}$. 
If  $B(0), \ldots, B(r-1)\in \mathcal{A}$, and $g=u_{r}s_{B(r-1)}u_{r-1}\cdots s_{B(0)}u_0$, then $g|_N$ is almost equal to $u_r\cdots u_0|_N$ or a constant function. 
\end{lem}
\proof 
If $A(i)= B(i)$ for all $i\in \{0,\ldots, r-1\}$, then  $g|_N$ is almost equal to $u_r\cdots u_0|_N$ since each $s_{B(i)}$ is the identity of $B(i)$. 

If $j\in \{0,\ldots, r-1\}$ is the least value such that $A(j)\not= B(j)$, then $u_js_{B(j-1)}u_{j-1}\cdots s_{B(0)} u_0|_N$ almost equals $u_j\cdots u_0|_N$
 as in the previous case. Since $\mathcal{A}$ is an almost disjoint family and $A(j)\not=B(j)$, it follows that $A(j)\cap B(j)$ is finite. But  $u_j\cdots u_1u_0(N)$ is almost contained in $A(j)$ and so
$$s_{B(j)}u_js_{B(j-1)}u_{j-1}\cdots s_{B(0)} u_0(n)=s_{B(j)}u_j\cdots u_1u_0(n)=0$$ for all but finitely many $n\in N$.
Therefore $g|_N$ is almost equal to a constant function.  \qed\vspace{\baselineskip}

\proofref{lost_monoid}
If $\mathcal{B}$ equals the union of $\A$ with the set of all finite subsets of $\nat$, then $S_{\mathcal{B}}$ is a semigroup equivalent 
to $S_{\A}$. Thus we may assume without loss of generality that $\A$ contains all finite sets and 
$S_{\A}$ is a subsemigroup of $\nat^\nat$.

It is clear that: 
$$\fin_2 \prec \fin_3 \prec \dots \prec \fin.$$ 
So it suffices to show that $\mathfrak{F}_2 \not \preccurlyeq S_{\mathcal{A}}$ and $S_{\mathcal{A}}\not \preccurlyeq \fin$. 
That $S_{\mathcal{A}}\not \preccurlyeq \fin$ follows since $\fin$ forms an ideal in $\nat^\nat$ and 
$|S_{\mathcal{A}} \setminus \fin| = 2^{\aleph_0}$. 

Let $U$ be any countable subset of $\nat^\nat$. We will show that $\fin_2\not\subseteq \genset{S_{\A}, U}$. Assume without loss of generality that $1_{\nat}\in U$. 
Partition $\nat$ into countably many infinite sets $N(u_0, \ldots, u_{m})$ indexed by the finite tuples $(u_0, \ldots, u_{m})\in U^{m+1}$ for all $m\in\nat$.  We shall 
define $f\in \fin_2$ such that 
$$f|_{N(u_0, \ldots, u_{m})}\not=u_ms_{A(m-1)}u_{m-1}\cdots s_{A(0)}u_0|_{N(u_0, \ldots, u_{m})}$$
for any $A(0), \ldots, A(m-1)\in \mathcal{A}$ whereby $f\not\in \genset{S_{\A}, U}$ and $\fin_2\not\preccurlyeq S_{\A}$. 

Let $u_0, \ldots, u_{m}\in U$ be arbitrary and let $N:=N(u_0, \ldots, u_{m})$. Let $r\in \{0,\ldots, m\}$ be the largest value such that $u_{r-1}\cdots u_0$ is almost injective on $N$ and $u_j\cdots u_0(N)$ is almost contained in some element of $\mathcal{A}$ for all $j\in\{0,\ldots, r-1\}$.  Such an $r$ exists since the conditions are vacuously satisfied when $r=0$.   
We will define $f|_{N}$ such that no extension of $f|_N$ to an element of $\nat^{\nat}$  lies in $u_m S_{\mathcal{A}} u_{m-1} \dots S_{\mathcal{A}} u_0$.
 If $g$ is any element of $u_m S_{\mathcal{A}} u_{m-1} \dots S_{\mathcal{A}} u_0$, then, by  applying Lemma \ref{2choices} to the string of factors from $u_r$ to $u_0$ in the expression for $g$, we see that $g|_N$ is almost equal to either:

\begin{enumerate}
\item[\rm (i)]  $(u_{m}s_{A(m-1)}u_{m-1}\cdots s_{A(r+1)}u_{r+1}s_{A(r)})(u_{r}\cdots u_0)$ for some $A(r), \ldots,  A(m-1)\in \mathcal{A}$; or
\item[\rm (ii)]  a constant function.
\end{enumerate}
From the definition of $r$ there are three cases to consider, since one of the following holds:
\begin{enumerate}
\item[\rm (a)]  $u_{r}\cdots u_0$ is not almost injective on $N$;
\item[\rm (b)]  $u_{m}\cdots u_0$ is almost injective on $N$ and $r=m$ ;
\item[\rm (c)]  $u_{r}\cdots u_0$ is almost injective on $N$, $r<m$, and $u_{r}\cdots u_0(N)$ is not almost contained in any set in $ \mathcal{A}$. 
\end{enumerate}

In each of these cases we shall construct $f|_N$ so that $f|_N$ is constant with value $1$ on some infinite coinfinite subset $M$ of $N$ and constant with value $0$ on $N\setminus M$. In any of these cases, if $g\in  u_m S_{\mathcal{A}} u_{m-1}
 \dots S_{\mathcal{A}} u_0$ and (ii) holds, then no matter how $M$ is defined $f|_N\not=g|_N$. Consequently, below we  verify that $f|_N\not=g|_N$ for  all $g\in  u_m S_{\mathcal{A}} u_{m-1}
 \dots S_{\mathcal{A}} u_0$ such that (i) holds. \vspace{\baselineskip}

\noindent{\bf Case (a).} Since $u_{r}\cdots u_0$ is not almost injective on $N$, there exist infinite disjoint sets $M=\set{m_i}{i\in\nat}\subseteq N$ and $\set{n_i}{i\in\nat}\subseteq N$ such that $u_{r}\cdots u_0(m_i)=u_{r}\cdots u_0(n_i)$ for all $i\in\nat$. In this case, we let $f|_N$ be defined by $f(m_i)=1$ and $f(n)=0$ for all $n\in N\setminus M\supseteq \set{n_i\in N}{i\in\nat}$.  If $g\in  u_m S_{\mathcal{A}} u_{m-1}\dots S_{\mathcal{A}} u_0$ and (i) holds, then $g(m_i)=g(n_i)$ for all but finitely many $i\in\nat$. Hence $f|_N\not=g|_N$, as required.
\vspace{\baselineskip}

\noindent{\bf Case (b).} In this case, we let $M$ be any infinite coinfinite subset of $N$ and
 define $f|_N$ so that $f(n)=1$ if $n\in M$ and $f(n)=0$ if $n\in N\setminus M$.  If $g\in  u_m S_{\mathcal{A}} u_{m-1}
 \dots S_{\mathcal{A}} u_0$ and (i) holds, then $g|_N$ almost equals $u_m\cdots u_0$ and so $g|_N$ is almost injective on $N$. But $f|_N$ is not almost injective on $N$ and so $g|_N\not= f|_N$.   \vspace{\baselineskip}

\noindent{\bf Case (c).} Since $u_{r}\cdots u_0(N)$ is not almost contained in any set in $ \mathcal{A}$, either there exists $A\in\mathcal{A}$ such that $u_{r}\cdots u_0(N)\cap A$ and $u_{r}\cdots u_0(N)\setminus A$ are infinite or $u_{r}\cdots u_0(N)\cap B$ is finite for all $B\in\mathcal{A}$. 
In the first case, let $M\subseteq N$ be such that both $u_{r}\cdots u_0(M)$ and $u_{r}\cdots u_0(N\setminus M)$ contain infinitely many points in $A$ and infinitely many points not in $A$. Then we define $f|_N$ so that $f(n)=1$ if $n\in M$ and $f(n)=0$ if $n\in N\setminus M$. If $g\in  u_m S_{\mathcal{A}} u_{m-1}
 \dots S_{\mathcal{A}} u_0$, (i) holds, and $A=A(r)$, then $g|_{(u_{r}\cdots u_0)^{-1}(\nat\setminus A)\cap N}$ is almost equal to a constant function. If $g\in  u_m S_{\mathcal{A}} u_{m-1}
 \dots S_{\mathcal{A}} u_0$, (i) holds, and $A\not=A(r)$, then $g|_{(u_{r}\cdots u_0)^{-1}(A)\cap N}$ is almost equal to a constant function.
In either case, $g|_N\not=f|_N$.

In the second case, i.e.,  $u_{r}\cdots u_0(N)\cap B$ is finite for all $B\in\mathcal{A}$,  we let $M$ be any infinite coinfinite subset of $N$. 
Then we define $f|_N$ so that $f(n)=1$ if $n\in M$ and $f(n)=0$ if $n\in N\setminus M$.  If $g\in  u_m S_{\mathcal{A}} u_{m-1}
 \dots S_{\mathcal{A}} u_0$ and (i) holds, then $g|_N$ is almost equal to a constant function while $f|_N$ maps infinitely many points to both $0$ and $1$. Hence $f|_N\not= g|_N$.
\qed


\section{Anti-chains}
 
  In \cite{Mesyan2007aa} it was proved that  $\preccurlyeq$ contains 
at least two incomparable elements  by constructing a subsemigroup $S$ of $\nat^\nat$ such that $S \not \preccurlyeq \fin_3$ and $\fin_3 \not \preccurlyeq S$. In Section \ref{almost_disjoint_section} we gave an example of a subsemigroup incomparable to all $\fin_n$.
The following theorem shows that there are anti-chains in $\preccurlyeq$ of arbitrary finite length. 

 \begin{thm}\label{finite_preorder} 
For all $i\in\nat$, there exist $i$ distinct closed subsemigroups contained in $\fin$ that are mutually incomparable under $\preccurlyeq$. 
\end{thm}

Let $m,k\in\nat$ be such that $m \geq 2$ and define $\U_{k,m}$ to be the semigroup of all $f \in \nat^\nat$ satisfying 
$$f(i) = i\text{ if  } i < k\text{ and }f(i)\in \{k, k+1,\ldots, k + m-1\}\text{ if }i\geq k.$$
It is easy to see that every $\U_{k,m}$ is a closed subsemigroup of $\nat^\nat$. Note that $\U_{0,m} \approx \fin_{m}$. 

\begin{lem}\label{mnk_theorem}
Let $k, l, m, n\in \nat$ be such that  $m,n \geq 2$. Then $\U_{k,m} \preccurlyeq \U_{l,n}$ if and only if $m \leq n$ and $k+m \leq l+n$.
\end{lem}
\proof 
($\Leftarrow$)
 We define $g,h \in \nat^\nat$ such that $\U_{k,m} 
\leq \genset{\U_{l,n},g,h}$.
Let $g,h \in \nat^\nat$ be  any mappings such that  
\begin{equation*}
g(i) = 
\begin{cases} 
i &\text{if } 0\leq i < k\\
i-(k+m)+(l+n) &\text{if } i\geq k
\end{cases}
\end{equation*}
\begin{equation*}
h(i)=
\begin{cases} 
i &\text{if } 0\leq i < k\\
i+(k+m)-(l+n) &\text{if }i\geq l+n-m.
\end{cases}
\end{equation*}
The mapping $h$ is well-defined since $l + n - m \geq k + m - m = k$.  Also, since $g(i)\geq l+n-m$ if $i\geq k$,   it follows that $hg=1_{\nat}$. 

Let $f \in \U_{k,m}$ be arbitrary and let $f' \in \nat^\nat$ be the map defined by 
\begin{equation*}
f'(i)=
\begin{cases}
i&\text{if }i< l+n-m\\
gfh(i)&\text{if }i\geq l+n-m.
\end{cases}
\end{equation*}
We prove that $f'\in  \U_{l,n}$.  If $i< l+n-m$, then  $f'(i)=i$ and, in particular, since $n\geq m$, $f'(j)=j$ for all $j<l$. 
If $i\geq l+n-m$, then $h(i)=i+(k+m)-(l+n)\geq k$. Hence $k\leq fh(i)\leq k+m-1$ and so $l\leq l+n-m\leq gfh(i)=f'(i)\leq l+n-1$.
Thus $f'\in \U_{l,n}$. 

To conclude, we show that $f=hf'g$. If $i< k$, then $hf'g(i)=hf'(i)=h(i)=i=f(i)$ since $k\leq l+n-m$. If $i\geq k$, then $g(i)\geq l+n-m$ and so $hf'g(i)=hgfhg(i)=f(i)$.
Therefore $f=hf'g$ and so $f\in \genset{\U_{l,n},g,h}$. Thus $\U_{k,m}\subseteq \genset{\U_{l,n},g,h}$ and so  $\U_{k,m} \preccurlyeq \U_{l,n}$.
\vspace{\baselineskip}

\noindent($\Rightarrow$)
We  prove the contrapositive.  If $k+m > l +n$, then  $\U_{k,m} \setminus \fin_{l+n}$ is uncountable. Since $\fin_{l+n}$ is an ideal in $\nat^\nat$, it follows that  $\U_{k,m} \not \preccurlyeq \fin_{l+n}$.
But $\U_{l,n} \subseteq \fin_{l+n}$ and therefore $\U_{k,m} \not \preccurlyeq \U_{l,n}$.

Now, assume that $m>n$.  Let $U$ be an arbitrary countable subset of $\nat^\nat$. We will show that $\U_{k,m}\not\subseteq \genset{\U_{l,n}, U}$. We may assume without loss of generality that $1_{\nat}\in U$. 
Let $\mathcal{K}\subseteq \mathcal{P}(\nat)$ be the set of finite unions of sets in $\set{f^{-1}(i)}{i\in\nat\text{ and } f\in \genset{U}}$ and let $f \in \genset{\U_{l,n}, U}$  be arbitrary. We will show that there are at most $n$ values  $i$ for which $f^{-1}(i)\not\in \mathcal{K}$. If $f \in \genset{U}$, then $f^{-1}(i)\in\mathcal{K}$ for all $i\in\nat$. Otherwise,
$$f=hgu$$ 
for some $u\in \genset{U}$, $g \in \U_{l,n}$, and $h\in \genset{\U_{l,n}, U}$. 
If $r \in \{0,1,\dots, l-1\}$, then 
 $$(gu)^{-1}(r)=u^{-1}(r) \in \mathcal{K}.$$ Hence $gu$ has at most $n$ preimages that are not in $\mathcal{K}$, namely the preimages of the elements $l,\ldots, l+n-1$. Every preimage of $f$ is a union of the preimages of $gu$ and, since $gu$ has finite image, it is a finite union. Hence any preimage of $f$ that is not in $\mathcal{K}$ must contain at least one of 
$(g u)^{-1}(l),\ldots,(g u)^{-1}(l+n-1)$. 
Thus $f$ has at most $n$ preimages that are not in $\mathcal{K}$.

On the other hand, we show that there exists $f\in \U_{k,m}$ with $m>n$ preimages that are not in $\mathcal{K}$. 
Since $\mathcal{K}$ is countable, there exists a partition $A_0,\dots, A_{m-1}$ of $\nat\setminus\{0,1,\ldots, k-1\}$ such that $A_0, \dots, A_{m-1}\not\in \mathcal{K}$. 
If $f$ is the element of $\U_{k,m}$ such that $f^{-1}(k+i)=A_i$ for all  $0\leq i \leq m-1$, then $f$ has the required property. It follows that $f \not \in \genset{\U_{l,n}, U}$ and so 
$\U_{k,m} \not \preccurlyeq \U_{l,n}$.
\qed\vspace{\baselineskip}

\proofref{finite_preorder} 
Let $i\in\nat$ be such that $i\geq 1$. We will show that the $i$ semigroups $\U_{0, i+1}, \U_{2, i}, \ldots,$ $\U_{2i-2,2}$ form an antichain under $\preccurlyeq$. Let $k, l, m, n\in \nat$ be such that $k+m=l+n=i+1$. Then, by Lemma \ref{mnk_theorem},   
$\U_{2k, m}\preccurlyeq  \U_{2l, n}$ if and only if $m \leq n$ and $2k+m \leq 2l+n$ if and only if $m=n$ and $k=l$ if and only if $\U_{2k,m}= \U_{2l, n}$. It follows that the semigroups $\U_{0, i+1}, \U_{2, i}, \ldots, \U_{2i-2,2}$ form an anti-chain in $\preccurlyeq$ of length $i$. 
\qed


\section{An uncountable chain}

A \emph{chain} inside a partial order is just a totally ordered subset. 

\begin{thm}\label{disjoint} 
There exists a chain, having length $\aleph_1$, of
$\approx$-classes containing (not necessarily closed) subsemigroups of $\fin_2$.
\end{thm}

If $A\subseteq \nat$, then we define $f_A\in\nat^\nat$ by
\begin{equation*}
f_A(i)=
\begin{cases}
1 &\text{if }i \in A\\
0 &\text{if } i\not\in A.
\end{cases}
\end{equation*}
If $\mathcal{A}\subseteq \mathcal{P}(\nat)$ containing $\emptyset$ or $\nat$, then write
$$F_{\mathcal{A}}=\set{f_A\in \nat^\nat}{A\in \mathcal{A}\text{ or }\nat\setminus A\in \mathcal{A}}.$$
It is easy to verify that $F_{\mathcal{A}}$ is a subsemigroup of $\C\leq \mathfrak{F}_2$. 

\begin{lem}\label{find_one}
Let $\mathcal{A}$ be a countable union of almost disjoint families $(\mathcal{A}_i)_{i\in \nat}$ of subsets of $\nat$ where $\A_i$ contains all finite subsets of $\nat$ for all $i\in\nat$, let $A$ be any infinite subset of $\nat$, and let $X$ be any countable subset of $\nat^\nat$. Then there exists $B\subseteq A$ such that  $f_B\not \in \genset{F_{\mathcal{A}}, X}$. 
\end{lem}
\proof 
Note that if $f_C\in F_{\A_i}$, $g\in \nat^{\nat}$, and $gf_{C}\in \C$, then $gf_C\in \{f_C, f_{\nat\setminus C}, f_{\nat}, f_{\emptyset}\}\subseteq F_{\A_i}$.
In particular, $F_{\mathcal{A}}=\bigcup_{i\in\nat} F_{\mathcal{A}_i}$ and $\C\cap\genset{F_{\A_i}, X}=\C\cap F_{\A_i}\genset{X}$. Hence
$$\C\cap \genset{F_{\mathcal{A}}, X}=\C\cap \genset{\bigcup_{i\in\nat} F_{\mathcal{A}_i}, X}=\C\cap \bigcup_{i\in\nat} F_{\mathcal{A}_i}\genset{X}$$
and so it suffices to find $f\in\C$ such that  $B:=f^{-1}(1)\subseteq A$ and $f\not\in F_{\mathcal{A}_i}\genset{X}$ for all $i\in\nat$.
Let $(U_{i,j})_{i,j\in\nat}$ be any infinite sets partitioning $A$ and let $\genset{X}=\{x_0, x_1,  \ldots\}$. We shall specify a subset $V_{i,j}$ of $U_{i,j}$ for all $i,j\in\nat$ such that if  $f\in \nat^\nat$ is any mapping such that 
\begin{equation*}
f(n)=
\begin{cases}
1 &\text{if }n \in V_{i,j}\\
0 &\text{if } n\in U_{i,j}\setminus V_{i,j},
\end{cases}
\end{equation*}
then $f\not\in  F_{\mathcal{A}_j}x_i$. 

If $x_i$ restricted to $U_{i,j}$ is not injective, then there exist
distinct $k,l \in U_{i,j}$ with $x_i(k)=x_i(l)$.
Thus if $g \in F_{\A_j}$, then $gx_i(k)=gx_i(l)$. In this case, we let $V_{i,j}$ be any subset of $U_{i,j}$ such that $k\in V_{i,j}$ and $l\not\in V_{i,j}.$

If $x_i$ is injective on $U_{i,j}$ and there exists $C\in \mathcal{A}_j$  such that $x_i(U_{i,j})\cap C$ is infinite, then  we define
$V_{i,j}$ to be any infinite coinfinite subset  of $U_{i,j}\cap x_i^{-1}(C)$. 
In this case, if $g\in  F_{\mathcal{A}_j}$, then $g x_i$ restricted to $U_{i,j}\cap x_i^{-1}(C)$ is almost equal to the constant function with value $0$ or $1$.  Hence $f\not\in  F_{\mathcal{A}_j}x_i$, as required.

If $x_i$ is injective on $U_{i,j}$ and $x_i(U_{i,j}) \cap C$ is finite for all $C\in\mathcal{A}_j$, then we define $V_{i,j}$ to be any infinite coinfinite subset of $U_{i,j}$. In this case,  as above,  if $g\in F_{\mathcal{A}_j}$, then $gx_i$ restricted to $U_{i,j}$ is almost equal to the constant function with value $0$ or $1$, and so $f\not\in  F_{\mathcal{A}_j}x_i$.

We complete the definition of $f$ by setting $f(n)=0$ for all $n\in\nat\setminus A$. 
From our construction, $f^{-1}(1)\subseteq A$ and $f\not\in  F_{\mathcal{A}_j}x_i$ for all $i,j\in\nat$, as required. 
\qed\vspace{\baselineskip}

\proofref{disjoint}
Let $\mathcal{A}_0$ be any almost disjoint family of cardinality $2^{\aleph_0}$ containing all the finite subsets of $\nat$.
Then for all countable $X\subseteq \nat^\nat$, by Lemma \ref{find_one}, there exists $f\in \C$ such that $f\not\in \genset{F_{\mathcal{A}_0}, X}$. In particular, $F_{\mathcal{A}_0}\prec \C\approx \fin_2$. 

We define by transfinite recursion a chain $(F_{\mathcal{A}_{\alpha}})_{\alpha<\aleph_1}$ such that $\mathcal{A}_{\alpha}$ is a countable union of almost disjoint families and $F_{\mathcal{A}_\alpha}\prec F_{\mathcal{A}_{\beta}}\prec \C$  for all ordinals $\alpha<\beta<\aleph_1$.

Assume that $\alpha<\aleph_1$ and that we have defined countable unions $\mathcal{A}_{\beta}$ of almost disjoint families for all $\beta<\alpha$.  Let $\mathcal{B}_{\alpha}=\bigcup_{\beta<\alpha} \mathcal{A}_{\beta}$, let $\mathcal{A}=(A_{\lambda})_{\lambda<2^{\aleph_0}}$ be an almost disjoint family of subsets of $\nat$, and let 
$(X_\lambda)_{\lambda<2^{\aleph_0}}$ be the countable subsets of $\nat^\nat$. 
Since every $\mathcal{A}_\beta$, $\beta<\alpha$, is a countable union of almost
 disjoint families and $\alpha$ is a countable ordinal, it follows that $\mathcal{B}_\alpha$ is a
countable union of almost disjoint families.
By Lemma \ref{find_one}, for all $\lambda<2^{\aleph_0}$ there exists $C_{\lambda}\subseteq A_{\lambda}$  such that $f_{C_\lambda}\not\in \genset{F_{\mathcal{B}_{\alpha}}, X_{\lambda}}$.  Let $\mathcal{A}_{\alpha}=\mathcal{B}_{\alpha}\cup \set{C_{\lambda}}{\lambda<2^{\aleph_0}}$.  Then 
$\set{C_{\lambda}}{\lambda<2^{\aleph_0}}$ is an almost disjoint family, since  if $\lambda\not=\lambda'$, then
 $C_{\lambda}\cap C_{\lambda'}\subseteq A_{\lambda}\cap A_{\lambda'}$ and the
 latter is finite since $\mathcal{A}$ is an almost disjoint family. Hence
$\mathcal{A}_{\alpha}$ is a countable  union of almost disjoint families.  In particular, by Lemma \ref{find_one}, $F_{\mathcal{A}_{\alpha}}\prec \C$. 
By construction, $F_{\mathcal{B}_{\alpha}}\leq F_{\mathcal{A}_{\alpha}}\not\preccurlyeq F_{\mathcal{B}_{\alpha}}$ and so $F_{\mathcal{B}_{\alpha}}\prec F_{\mathcal{A}_{\alpha}}$.  
It follows that $F_{\mathcal{A}_{\beta}}\leq F_{\mathcal{B}_{\alpha}}\prec F_{\mathcal{A}_{\alpha}}$ for all $\beta<\alpha$. \qed


\section*{Acknowledgements}

The work of M. ~Morayne has been partially financed by NCN means granted by decision DEC-2011/01/B/ST1/01439.  J. D. 
Mitchell and Y. P\'eresse would like to thank the University of Colorado at Colorado Springs for their hospitality during the writing of 
this article. We are grateful to the referee for comments and suggestions that have led to significant improvements in the paper.



\end{document}